\documentclass[11pt]{amsart}
\usepackage[all]{xy}
\usepackage{nicefrac}
\usepackage{amsmath}
\usepackage{amssymb}

\swapnumbers

\theoremstyle{plain}
\newtheorem{prop}[subsection]{Proposition}
\newtheorem{sublemm}[subsubsection]{Lemma}

\newtheorem{theo}[subsection]{Theorem}
\newtheorem{coro}[subsection]{Corollary}

\def\t{\otimes}
\def\fdr{\rightarrow}

\def\PP{\mathsf P}
\def\QQ{\mathsf Q}

\def\PMQ{{}^{}_\PP \mathcal M^0_\QQ}

\def\KK{\mathbb K}
\def\NN{\mathbb N}
\def\ZZ{\mathbb Z}

\def\we{\stackrel{\sim}{\rightarrow}}

\DeclareMathOperator{\Span}{Span}

\DeclareMathOperator*{\colim}{colim}

\title[$\Gamma$-Homology of algebras over an operad]{Gamma-Homology of algebras over an operad}
\author{Eric Hoffbeck}
\address{Laboratoire Paul Painlev\'e, Universit\'e de Lille 1, Cit\'e Scientifique, 
59655 Villeneuve d'Ascq Cedex, France}
\email{Eric.Hoffbeck@math.univ-lille1.fr}

\begin{document}

\begin{abstract}
The purpose of this paper is to study generalizations of Gamma-homology in the context of operads. 
Good homology theories are associated to operads under appropriate cofibrancy hypotheses, but this requirement is not satisfied by usual operads outside the characteristic zero context. In that case, the idea is to pick a cofibrant replacement $\QQ$ of the given operad $\PP$. We can apply to $\PP$-algebras the homology theory associated to $\QQ$ in order to define a suitable homology theory on the category of $\PP$-algebras. We make explicit a small complex to compute this homology when the operad $\PP$ is binary and Koszul. In the case of the commutative operad $\PP = \mathsf{Com}$, we retrieve the complex introduced by Robinson for the Gamma-homology of commutative algebras.
\end{abstract}
    
\maketitle

The classical homology theories of commutative algebras (Harrison homology in the differential graded setting over a field of characteristic $0$, cf. \cite{Harrison}, Andr\'e-Quillen homology in the simplicial setting over a ring of any characteristic, cf. \cite{Quillen} and \cite{Andre}) can be considered as homology theories associated to the commutative operad $\mathsf {Com}$.
There is another homology theory for commutative algebras, $\Gamma$-homology (Gamma-homology in plain words, also called topological Andr\'e-Quillen), which has been introduced by Robinson and Whitehouse in \cite{RW}, and by Basterra in \cite{Bast} (with a different point of view), to solve obstruction problems in homotopy theory.  In the setting of \cite{RW}, Gamma-homology is defined as the homology theory associated to an $E_\infty$-operad (a cofibrant replacement of $\mathsf {Com}$).
This new homology can be defined in the context of differential graded or simplicial context or in the context of spectra, and gives the same result in each case (cf. Mandell \cite{Man}), in contrast with the usual Andr\'e-Quillen homology.

The purpose of this paper is to study generalizations of $\Gamma$-homology in the context of operads.

Usual methods of homotopical algebra apply to the categories of algebras associated to operads  which are cofibrant, or at least which fulfill sufficiently strong cofibrancy requirements.
As a consequence, we have a good homology theory $H_*^\QQ$ associated to any such operad $\QQ$. 
But many usual operads, like the commutative operad $\mathsf {Com}$ or the Lie operad $\mathsf {Lie}$, do not fit this framework (unless we work with differential graded modules over a field of characteristic $0$). In this situation, a natural idea is to pick a cofibrant replacement of the given operad $\PP$, let $\QQ \we \PP$, and to apply the homology $H_*^\QQ$ to $\PP$-algebras  in order to obtain a consistent homology theory on the category of $\PP$-algebras.
We use the notation $H\Gamma_*^\PP = H_*^\QQ$ and the name $\Gamma$-homology to refer to this homology theory after observing that different choices of $\QQ$ give the same result.

This generalizes the usual notion of $\Gamma$-homology where $\PP =\mathsf{Com}$ and $\QQ$ is an $E_{\infty}$-operad. The homology $H_*^\QQ$ associated to a cofibrant replacement of the operad $\mathsf{Lie}$ has also been used by Chataur, Rodriguez and Scherer in \cite{CRS}. 

The problem is that the choice of a cofibrant replacement is satisfying in theory, but making such a cofibrant replacement explicit is often very difficult (especially when the ground ring is not a field of characteristic $0$).
We give a direct definition of  $H\Gamma_*^\PP$, which agrees with the initial one, but  where the choice of an operadic cofibrant replacement is avoided. The idea is to use the model category on $\PP$-bimodules, which only needs mild assumptions on $\PP$. We  show how to define a complex to determine $H\Gamma_*^\PP$ from a choice of a cofibrant replacement of the operad $\PP$, not in the category of operads, but in the category of $\PP$-bimodules, the operad $\PP$ being viewed as a bimodule over itself. The category of $\PP$-bimodules is easier to deal with than the category of operads.

In \cite{Robinson}, Robinson makes explicit a small complex, analogous to Harrison's complex, which computes usual $\Gamma$-homology.
In the case where the operad $\PP$ is Koszul, we define an explicit complex to compute the $\Gamma$-homology associated to $\PP$.
Recall that an operad is Koszul if we have a quasi-isomorphism between $(\PP \circ K\PP \circ \PP, \partial)$ and $\PP$, where $K\PP$ is the Koszul construction, defined by $K(\PP)_{(s)}:=H_s(B_*(\PP)_{(s)}, \partial)$. In \cite{Balav}, Balavoine defined  a complex computing $H_*^\PP$ when working over a field of characteristic $0$, using the Koszul construction.
When we work over a ring of any characteristic, finding a complex is more complicated, as we need to resolve the symmetries in $K\PP$. This can be done by tensoring the Koszul construction by the acyclic bar construction of the symmetric group. 
Finally, we get a small explicit complex computing the $\Gamma$-homology of $\PP$-algebras.
Our construction coincides with Robinson's complex for the case $\PP = \mathsf {Com}$. 
As an illustration, we make our complex explicit in the case $\PP = \mathsf{Lie}$.

We also define a cohomology theory $H\Gamma^*_\PP$ associated to any operad $\PP$.

\vspace{0.5cm}

In Section 1, we recall the model category structures we use in the paper: dg-modules, $\Sigma_*$-modules, bimodules, algebras over operads.  Most of the model structures we consider are defined by a transfer of structure. We make the cofibrations explicit in each case.
In a second part, we recall the usual notion of homology for algebras over a cofibrant operad, and show how to reduce the complex when we are given a cofibrant replacement of bimodules. Then we make a similar construction of a reduced complex when the operad is not cofibrant. This leads us to the definition of $\Gamma$-homology of algebras over an operad (without any cofibrancy hypothesis). 
In Section 3, we construct an explicit complex for any binary Koszul operad $\PP$ to compute $\Gamma$-homology. This complex is defined using the Koszul construction $K\PP$ and the acyclic bar construction of the symmetric group.

\vspace{0.5cm}
\textbf{Convention.} 
We work in the differential graded setting. 
We take a category of differential graded modules (for short dg-modules) over a fixed base ring $\KK$ as a base category (see Section \ref{modcatdg} for details). We use the letter $\mathcal C$ to denote this category. When necessary, we assume tacitely that any dg-module, and more generally that any object defined over this base category, consists of projective modules over the ground ring.

We review the definition of the model category of $\Sigma_*$-modules underlying the category of operads in Section \ref{modcatsigma}, the model category of bimodules in Section \ref{modcatbimod}.
All operads $\PP$ will be assumed to be connected, in the sense that $\PP(0)=0$ and $\PP(1)=\KK$. All $\Sigma_*$-modules $M$, and more generally any object defined over the category of $\Sigma_*$-modules, will be assumed to be connected, that is $M(0)=0$.

\vspace{2cm}

\section{Model categories}
We review here the model structures for the categories which are used in this paper.
For general references on the subject, we refer the reader to the survey of Dwyer and Spalinksi \cite{DS} and the books of Hirschhorn \cite{Hirsch} and Hovey \cite{Hovey}. For model structures in the operadic context, we refer to the articles of Hinich \cite{Hinich} and of Goerss and Hopkins \cite{GH}, and the book of Fresse \cite{BouquinBenoit}.

\subsection{Transfer of structure}\label{transf}
We use the notion of a pair of adjoint functors to transport model structures.  
Suppose we have an adjunction $$F : \mathcal{X} \rightleftarrows \mathcal{A} : U$$ such that $\mathcal{X}$ is a cofibrantly
generated model category and $\mathcal{A}$ is a category equipped with colimits and limits. We can then define classes of weak equivalences, fibrations and cofibrations in $\mathcal{A}$.
\begin{itemize}
 \item The weak equivalences in $\mathcal A$ are morphisms $f$ such that $U(f)$ is a weak equivalence in $\mathcal X$.
 \item The fibrations in $\mathcal A$ are morphisms $f$ such that $U(f)$ is a fibration in $\mathcal X$.
 \item The cofibrations are the morphisms which have the left lifting property (in short, LLP) with respect to acyclic fibrations.
\end{itemize}

Under some technical hypotheses (cf. \cite[Theorem 11.3.2]{Hirsch}), a classical result says that $\mathcal A$ is equipped with a model structure given by the weak equivalences, fibrations and cofibrations above. Under weaker hypotheses (cf. \cite[Theorem 12.1.4]{BouquinBenoit}), the category $\mathcal A$ is equipped with a semi-model category, that is the lifting and factorization axioms only hold when the morphisms have a cofibrant domain. Semi-model categories will be enough for us here.

We can describe the generating (acyclic) cofibrations of the semi-model category $\mathcal A$ explicitely: they are the morphisms $F(i) : F(C) \fdr F(D)$ such that $i$ ranges over the generating (acyclic) cofibrations of $\mathcal X$.

\subsection{Model category structure for dg-modules}\label{modcatdg}
In this paper, the dg-modules we consider are $\ZZ$-graded modules endowed with a differential $\delta$ decreasing the degree by $1$. The category of dg-modules is denoted by $\mathcal C$. 
The internal hom of this category is denoted by $Hom_{\mathcal C}(C,D)$, for all $C,D\in\mathcal C$. This dg-module is spanned in degree $d$ by the linear maps $f: C\rightarrow D$ which raises degrees by $d$. The differential of such a map in $Hom_{\mathcal C}(C,D)$ is defined by its graded commutator with the internal differential of $C$ and $D$. We adopt the terminology of homomorphisms to distinguish the elements of the dg-hom $Hom_{\mathcal C}(C,D)$ from the actual morphisms of dg-modules, the linear maps which preserve gradings and commute with differentials.

The category of dg-modules is equipped with its usual model structure:
The weak equivalences are the quasi-isomorphisms and the fibrations are degreewise surjective maps (cofibrations are characterized by the LLP with  respect to acyclic fibrations). 

Let $D_n = \KK d_n \oplus \KK c_{n-1}$  where $d_n$ is a homogeneous element in degree $n$ sent by the differential to $c_{n-1}$ in degree $n-1$. Let $C_n$ be $\KK c_{n-1}$, submodule of $D_n$. The embeddings  $C_n \fdr D_n$, $n \in \ZZ$, define a set of generating cofibrations in $\mathcal C$. The maps $0 \fdr D_n$ define generating acyclic cofibrations.

In what follows, the underlying dg-module of any object is tacitely assumed to be cofibrant. 

\subsection{Twisted dg-modules}
In general, we assume that a dg-module $C$ is equipped with a differential $\delta : C \fdr C$. We sometimes twist this internal differential by a cochain $\partial \in Hom_{\mathcal C}(C,C)$ of degree $-1$ in order to get a new differential $\delta + \partial$. We assume the relation $\delta \circ \partial + \partial \circ \delta + \partial^2=0$, in order to obtain that $\delta+\partial$ satisfies $(\delta+\partial)^2=0$.
We usually omit the internal differential $\delta$ in the notation: We write $C$ for the module $C$ with differential $\delta$ and write $(C, \partial)$ to denote the module $C$ with differential $\delta+\partial$. 

We are going to define quasi-free objects (algebras over operads, bimodules), twisted objects $(C, \partial)$ such that $C$ is free with respect to an algebraic structure.

\subsection{Model category structure for $\Sigma_*$-modules}\label{modcatsigma}
We use the notation $\mathcal M$ for $\Sigma_*$-modules. We have an adjunction between the forgetful functor $U$ (from the category  $\mathcal M$ to the category of chain complexes) and the free $\Sigma_*$-module functor $\Sigma_* \t -$.

The transfer process of Section \ref{transf} gives us a model structure on $\Sigma_*$-modules where weak equivalences are morphisms whose all components are weak equivalences of dg-modules and where fibrations are morphisms whose all components are epimorphisms of dg-modules. Cofibrations are given by the LLP with respect to acyclic fibrations.
Again, we can say more precisely which maps are cofibrations (cf. \cite[Prop. 11.4.A]{BouquinBenoit}). The generating cofibrations are given by tensor products 
$$i \t F_r : C \t F_r \fdr D \t F_r$$
where $i : C \fdr D$ ranges over the generating cofibrations of dg-modules and 
$F_r, r\in\NN$ denote the $\Sigma_*$-modules such that 
$$F_r(n)= \left \{ \begin{array}{rl} \KK[\Sigma_r], & \text{for $n=r,$} \\ 0, & \text{otherwise.} \end{array} \right. $$

We will use the composition product $\circ$ of $\Sigma_*$-modules.
Recall that for a constant $\Sigma_*$-module $N$ (such that $N(0)=C$ and $N(r)=0$ for $r>0$), the composition $M \circ C$ represents the application of a symmetric functor with coefficients in $M$ to $C$:
$$M \circ C = \bigoplus_{r=1}^{+\infty} (M(r) \t C^{\t r})_{\Sigma_r}.$$
This object is denoted by $S(M,C)$ in the book \cite{BouquinBenoit}, but we use the notation $M \circ C$ in the paper.

In general, the composite $M \circ N$ is defined such that the associativity relation $M \circ (N \circ C)= (M \circ N) \circ C$ is satisfied for all constant $\Sigma_*$-modules $C$. The composition product $\circ$ is a monoidal product for the category of $\Sigma_*$-modules.

Recall that an operad is a $\Sigma_*$-module $\PP$ equipped with an initial morphism $\mathsf{I} \fdr \PP$ (where $\mathsf{I}$ is the unit $\Sigma_*$-module, with $\KK$ in arity $1$ and $0$ everywhere else) and a composition product $\gamma : \PP \circ \PP \fdr \PP$. As mentionned in the introduction, we assume that any operad $\PP$ satisfies $\PP(0) = 0$ and $\PP(1) = \KK$, so that the initial morphism of $\PP$ is an isomorphism in arity $0$ and in arity $1$. 
We use the notation $\bar{\PP}$ for the $\Sigma_*$-submodule of $\PP$ formed by the components $\PP(n)$ of arity $n>1$ and trivial in arity $0$ and in arity $1$.

In what follows, we will often consider $\Sigma_*$-cofibrant operads, operads $\PP$ such that the initial morphism $\mathsf{I} \fdr \PP$ is a cofibration of $\Sigma_*$-modules.

\subsection{Model category structure for bimodules over operads}\label{modcatbimod}
Let $\PP$ and $\QQ$ be operads. Let $\PMQ$ be the category of connected (that is $M(0)=0$) $\PP$-$\QQ$-bimodules in the sense of \cite{BouquinBenoit}. We have an adjunction $$\PP \circ - \circ \QQ :  \mathcal M \rightleftarrows \PMQ : U,$$ where $U$ is the forgetful functor.

The transfer process  gives us a semi-model structure on $\PP$-$\QQ$-bimodules, where weak equivalences are morphisms whose all components are weak equivalences of dg-modules and where fibrations are morphisms whose all components are epimorphisms of dg-modules. Cofibrations are given by the LLP with respect to acyclic fibrations.

We now describe a particular class of cofibrant $\PP$-$\QQ$-bimodules that we will use extensively later.

\begin{prop}\label{cofbimod}
Let $\PP$ and $\QQ$ be connected operads and $M$ a cofibrant $\Sigma_*$-module.

The quasi-free $\PP$-$\QQ$-bimodule $(\PP \circ M \circ \QQ, \partial)$ is cofibrant if the differential is decomposable (that is $\partial(M) \subset \bar \PP \circ M \circ \QQ + \PP \circ M \circ \bar \QQ$).
\end{prop}

This result will be deduced from the following lemmas.
\begin{sublemm}\label{cofbimod1}
Let $\QQ$ be a connected operad and $M$ a cofibrant $\Sigma_*$-module.

The quasi-free $\QQ$-module $(M \circ \QQ, \partial)$ is cofibrant if the differential is decomposable (that is $\partial (M) \subset M \circ \bar \QQ$).
\end{sublemm}

\begin{proof}
The complex $(M \circ \QQ, \partial)$ is filtered by 
$$ar_\lambda(M \circ \QQ, \partial) = (ar_\lambda M \circ \QQ, \partial) $$
where $ar_\lambda M(n) = M(n)$ if $n \leq \lambda$ and $0$ otherwise, and where the differential  is just the restriction of the differential on $(M \circ \QQ, \partial)$.

Note that $\partial (ar_\lambda M) \subset ar_{\lambda-1}(M \circ \QQ)$.

We have the following pushout of right $\QQ$-modules:
\begin{equation*}\begin{array}{c}
 $\xymatrix@M=8pt@C=30pt@R=30pt{
(\partial M(n) \circ \QQ, 0) \ar[d] \ar[r] & (ar_{n-1} M \circ \QQ, \partial) \ar[d]\\
(\partial M(n) \circ \QQ \bigoplus M(n) \circ \QQ, \partial) \ar[r] & (ar_n M \circ \QQ, \partial) 
}$\end{array}
\end{equation*}

The arrow on the left is a generating cofibration. Thus the arrow on the right is a cofibration too.

Thus $\displaystyle{(M \circ \QQ, \partial) = \colim _\lambda ar_\lambda(M \circ \QQ, \partial)}$ is a cofibrant right $\QQ$-module.
\end{proof}

\begin{sublemm}\label{cofbimod2}
Let $N = (M \circ \QQ, \partial)$ be a right $\QQ$-module with the hypothesis of the above lemma.
Let $\PP$ be a connected operad.

The quasi-free $\PP$-$\QQ$-bimodule $(\PP \circ N, \partial)$ is cofibrant if the differential is decomposable (that is $\partial (N) \subset \bar\PP \circ N$).
\end{sublemm}

\begin{proof}
First, note that  $\partial (ar_\lambda N) \subset \bar \PP \circ ar_{\lambda -1} N$.
Therefore we can define a filtration by $ar_\lambda (\PP \circ N, \partial) = (\PP \circ ar_\lambda N, \partial)$.

Note that $\partial (ar_\lambda N) \subset ar_{\lambda-1}(\PP \circ N)$.
Using a similar argument as in the above proof, the obvious arrow $ar_{\lambda-1}(\PP \circ N, \partial) \fdr ar_\lambda(\PP \circ N, \partial)$ is a cofibration of $\PP$-$\QQ$-bimodules.

Thus $\displaystyle{(\PP \circ N, \partial) = \colim_\lambda ar_\lambda(\PP \circ N, \partial)}$ is a cofibrant $\PP$-$\QQ$-bimodule. 
\end{proof}
The combination of these two lemmas proves Proposition \ref{cofbimod}. \qed

\subsection{Model category structure for algebras over an operad}
We have an adjunction between the forgetful functor $U$ from $\PP$-algebras to dg-modules and the free $\PP$-algebra functor $\PP \circ -$.

If $\PP$ is $\Sigma_*$-cofibrant, the transfer process of Section \ref{transf} gives us a semi-model category on $\PP$-algebras, where weak equivalences are morphisms which are weak equivalences of dg-modules and where fibrations are morphisms which are epimorphisms of dg-modules. Cofibrations are given by the LLP with respect to acyclic fibrations.

The model category structure allows us to define the cofibrant replacement of a $\PP$-algebra $A$. It is a cofibrant $\PP$-algebra $Q_A$ such that we have a weak equivalence of $\PP$-algebras $Q_A \we A$.

If we are given a cofibrant replacement $_\PP Res_\PP \we \PP$ in the category of $\PP$-bimodules, we can easily make explicit a  cofibrant replacement of a $\PP$-algebra A.

First we need to recall the definition of the relative composition product of $\PP$-modules. Suppose that $M$ is a right $\PP$-module and $A$ a $\PP$-algebra. We denote by $M \circ_P A$ the quotient of $M \circ A$ coequalizing the right action of $\PP$ on $M$ and the left action of $\PP$ on $A$. When $M$ is a $\PP$-bimodule, the relative composite $M \circ_P A$ inherits a $\PP$-algebra structure.

We can now give the result

\begin{sublemm}\label{RempCofib}
 We get a cofibrant replacement $(_\PP Res_\PP \circ_\PP A, \partial')$ of $A$ in the category of $\PP$-algebras, with $\partial'=\partial \circ_\PP A.$ 
\end{sublemm}

\begin{proof}
The $\PP$-algebra $(_\PP Res_\PP \circ_\PP A, \partial')$ is cofibrant, following the same argument of the proof of Lemma \ref{cofbimod1}.
The $\PP$-bimodule $_\PP Res_\PP$ is cofibrant, thus it is cofibrant as a right $\PP$-module. The operad $\PP$ is also cofibrant as a right $\PP$-module. As the functor $- \circ_\PP A$ preserve weak equivalences between cofibrant objects (cf. \cite[Theorem 15.1.A]{BouquinBenoit}), we get that $(_\PP Res_\PP \circ_\PP A, \partial')$ is a cofibrant replacement of $\PP \circ_\PP A$. But $\PP \circ_\PP A = A$, thus $(_\PP Res_\PP \circ_\PP A, \partial')$ is a cofibrant replacement of $A$ in the category of $\PP$-algebras.
Explicitely, the differential $\partial'$ is given by $\partial'(m \circ (a_1,\ldots, a_n)) = (\partial(m)) \circ (a_1, \ldots, a_n)$ where $m$ lies in $_\PP Res_\PP$ and $\partial(m)$ in $\PP \circ ({}_\PP Res_\PP) \circ \PP$. Note that we use the structure of $\PP$-algebra of $A$ on the right hand side to get an element of $_\PP Res_\PP \circ_\PP A$.
\end{proof}

\vspace{2cm}

\section{Gamma-homology of $\PP$-algebras}
In this section, we recall the definition of the homology of $\QQ$-algebras for $\QQ$ a $\Sigma_*$-cofibrant operad. In the differential graded setting over a ring of characteristic $0$, homology with trivial coefficients was defined by Getzler and Jones in \cite{GJ} and homology with coefficients was defined by Balavoine in \cite{Balav}. The extension to any category of dg-modules can be found in \cite{Hinich}. We adopt conventions of \cite{BouquinBenoit} where these notions are reviewed.
We define $\Gamma$-homology of $\PP$-algebras for any operad $\PP$, using bimodule resolutions. Then we prove the identity $H_\QQ = H\Gamma^*_\PP$ when $\QQ$ is a $\Sigma_*$-cofibrant replacement of $\PP$.

\subsection{Recollections on homology of $\QQ$-algebras}\label{Qhomol}
We refer the reader to Section 4 of \cite{BouquinBenoit} for the first definitions.

Let $\QQ$ be a $\Sigma_*$-cofibrant operad, $B$ an algebra over $\QQ$.

We denote by $U_{\QQ}(B)$ the enveloping algebra of $B$ and by $\Omega_{\QQ}(B)$ the module of K\"ahler differentials of $B$. 

The enveloping algebra $U_{\QQ}(B)$ is spanned by elements $q(\diamond, b_1, \ldots, b_n)$, where $q \in \QQ(n+1)$,  $b_1, \ldots, b_n \in B$ and the symbol $\diamond$ denotes a free input, divided out by the relations
$$p(\diamond, b_1, \ldots, b_{i-1}, q(b_i, \ldots, b_n), b_{n+1}, \ldots, b_m)=p \circ_{i+1} q (\diamond, b_1, \ldots, b_{i-1}, b_i, \ldots, b_m).$$

The product is given by 
$$p(\diamond, a_1, \ldots, a_n) . q(\diamond, b_1, \ldots, b_m) = p \circ_1 q (\diamond, b_1, \ldots, b_m, a_{1}, \ldots, a_{n}).$$
We can put $\diamond$ at any place since the action of $\Sigma_{n+1}$ on $\QQ(n+1)$ allows us to permute the inputs of any operation $q \in \QQ(n+1)$.

We represent graphically an element $q(\diamond, b_1, \ldots, b_n)$ by 
$$\vcenter
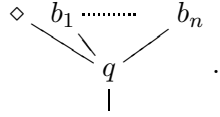
{\xymatrix@M=3pt@H=4pt@W=3pt@R=8pt@C=4pt{
\diamond\ar@{-}[drr] & b_1\ar@{-}[dr]\ar@{.}[rr] &&& b_n\ar@{-}[dll] \\
&& q\ar@{-}[d] && \\
&&}  }.$$

The module of K\"ahler differentials $\Omega_{\QQ}(B)$ is a left module over $U_\QQ(B)$ such that 
$$Hom_{U_\QQ(B)}(\Omega_{\QQ}(B), F) = \textrm{Der}_\QQ(B,F) $$
for all left modules $F$ over $U_\QQ(B)$, where $\textrm{Der}_\QQ(B,F)$ denotes the dg-module of $\QQ$-derivations $B \fdr F$ (not necessarily preserving the degree) and where $Hom_{U_\QQ(B)}(\Omega_{\QQ}(B), F)$ is the dg-module of homomorphisms of left $U_\QQ(B)$-modules between $\Omega_{\QQ}(B)$ and $F$.

The module of K\"ahler differentials $\Omega_{\QQ}(B)$ can be seen as the dg-module spanned by elements $q (b_1, \ldots, d b_i, \ldots, b_n)$, where $q \in \QQ(n)$, $b_1, \ldots, b_n \in B$ and $d$ denotes a formal differentiation symbol, divided out by the relations
\begin{multline*}
p(b_1, \ldots, q(b_i, \ldots, b_n), b_{n+1}, \ldots, db_j, \ldots, b_m) \\
= p \circ_i q (b_1, \ldots, b_i, \ldots, b_n, \ldots, db_j, \ldots, b_m), \text{for } i\neq j, 
\end{multline*}
\begin{multline*}
p(b_1, \ldots, dq(b_i, \ldots, b_n), b_{n+1}, \ldots, \ldots, b_m) \\
= \sum_{j=i}^{n} p \circ_i q  (b_1, \ldots, b_i, \ldots, db_j, \ldots, b_n, \ldots, b_m).
\end{multline*}

Let us now define the homology and the cohomology of an algebra over~$\QQ$.

We choose $Q_B$ a cofibrant replacement of $B$. Let $E$ be a right $U_\QQ(Q_B)$-module and $F$ be a left $U_\QQ(Q_B)$-module.

The homology of $B$ as a $\QQ$-algebra with coefficients in $E$ is defined by $H_*^\QQ(B,E)= H_*(E \t_{U_{\QQ}(Q_B)} \Omega_{\QQ}(Q_B))$.
In a similar way, the cohomology of $B$ is defined by $H^*_\QQ (B,F)= H^*(Hom_{U_\QQ(Q_B)}(\Omega_{\QQ}(Q_B), F))$.

\vspace{1cm}
We will use the following lemma to reduce the complex appearing in the calculation of the homology and the cohomology.

\begin{sublemm}\label{lemmreduc}
If $Q_A$ is a quasi-free $\QQ$-algebra $Q_A=(\QQ(C), \partial')$, then we have an isomorphism of left $U_\QQ(Q_A)$-modules 
$$(U_\QQ(Q_A) \t C, \partial'') \simeq \Omega_\QQ (Q_A) $$
where the differential $\partial'' : U_\QQ(Q_A) \t C \fdr U_\QQ(Q_A) \t C$ is a twisting homomorphism on $U_\QQ(Q_A) \t C$-modules induced by the action of the twisting derivation of $Q_A$ on $U_\QQ(Q_A) \t C$ (see the detailed representation in  Figure~\ref{figdiff}).
\end{sublemm}

\begin{proof}
We begin to prove the result for a free algebra $Q_A=\QQ(C)$.\\
First, we have $\textrm{Der}_\QQ(\QQ(C),F) = Hom_{\mathcal C}(C,F)$. To prove this identification, we define $\Phi : \textrm{Der}_\QQ(\QQ(C),F) \fdr Hom_{\mathcal C}(C,F)$ by $\Phi(\theta) = \theta_{|C} : C \fdr F$. This is application an isomorphism.
The inverse map $\Phi^{-1}$ associates to any $f : C \fdr F$ the derivation $\theta_f$ such that $\theta_f (q(c_1, \ldots, c_n)) = \pm \sum_i q(c_1, \ldots, f(c_i), \ldots, c_n)$ where the signs are induced by the usual Koszul rule. \\

We have $Hom_{\mathcal C}(C,F) = Hom_{U_\QQ(Q_A)}(U_\QQ(Q_A) \t C,F)$, which gives us $\textrm{Der}_\QQ(Q_A,F) = \textrm{Mor}_{U_\QQ(Q_A)}(U_\QQ(Q_A) \t C,F)$.\\
But $\Omega_{\QQ}(Q_A)$ is defined by $Hom_{U_\QQ(Q_A)}(\Omega_{\QQ}(Q_A), F) = \textrm{Der}_\QQ(Q_A,F)$.
Thus Yoneda's lemma gives us an isomorphism $\Psi$ of $U_\QQ(Q_A)$-modules between  $U_\QQ(Q_A) \t C$ and $\Omega_\QQ (Q_A).$\\

The map $\Psi : U_\QQ(Q_A) \t C \fdr \Omega_\QQ (Q_A)$  associates to the element $q(\diamond, a_1, \ldots, a_n) \t c$ the element $q(dc, a_1, \ldots, a_n)$.

\begin{figure}
\begin{eqnarray*} 
	\vcenter{\xymatrix@M=3pt@H=4pt@W=3pt@R=8pt@C=4pt{
 	c_1 \ar@{-}[dr]\ar@{.}[rr] & \save[]!C.[d]!C *+<36pt,10pt>\frm{}*+<6pt>\frm{(}*+<0pt>\frm{)}
\restore\ar@{}[]!L-<28pt,5pt>;[d]!L-<28pt,5pt>_(0.3){\displaystyle{\partial}} & c_n\ar@{-}[dl] \\
 	& q_0\ar@{-}[drrr] &&& a_1\ar@{-}[d] \ar@{.}[rrr] &&& a_r\ar@{-}[dlll] & \\
	&&&& q\ar@{-}[d] &&&& \\
	&&&&&&&&}}  
& = &  	\sum_{i=1}^n \vcenter{\xymatrix@M=3pt@H=4pt@W=3pt@R=8pt@C=4pt{
 	c_1\ar@{.}[r]\ar@{-}[dr] & dc_i\ar@{.}[r]\ar@{-}[d] & c_n\ar@{-}[dl] \\
&q_0\ar@{-}[dr]& & a_1\ar@{-}[dl] \ar@{.}[r] & a_r\ar@{-}[dll] & \\
	&& q\ar@{-}[d] && \\
	&&&&}}   \\
& = &	\sum_{i=1}^n \vcenter{\xymatrix@M=3pt@H=4pt@W=3pt@R=8pt@C=4pt{
 	c_1\ar@{.}[r]\ar@{-}[drr] & dc_i\ar@{.}[r]\ar@{-}[dr] & c_n\ar@{-}[d] & a_1\ar@{-}[dl] \ar@{.}[r] & a_r\ar@{-}[dll] & \\
	&& q\ar@{-}[d] \circ_1 q_0 && \\
	&&&&}}   \\
& \stackrel{\Psi^{-1}}{\mapsto} & 
	\sum_{i=1}^n \vcenter{\xymatrix@M=3pt@H=4pt@W=3pt@R=8pt@C=4pt{
  	c_1\ar@{.}[r]\ar@{-}[drr] & \diamond \ar@{.}[r]\ar@{-}[dr] & c_n\ar@{-}[d] & a_1\ar@{-}[dl] \ar@{.}[r] & a_r\ar@{-}[dll] & \\
 	&& q\ar@{-}[d] \circ_1 q_0 && \\
	&&&&}} \t c_i 
\end{eqnarray*}
where the empty box is the $i$-th input of the tree.
\caption{\footnotesize{A graphical representation of the inverse isomorphism $\Psi^{-1}$.}}
\label{figPsi} \end{figure}

Its inverse  $\Psi^{-1} : \Omega_\QQ (Q_A) \fdr U_\QQ(Q_A) \t C$ sends \\
$q(d q_0 (\underline{c_0}), a_1, \ldots, a_r) = \sum_i q \circ_1 q_0 (c_1, \ldots, dc_i, \ldots, c_n , a_1, \ldots, a_r)$ \\
to $\sum_i q \circ_1 q_0 (c_1, \ldots, \diamond, \ldots, c_n , a_1, \ldots, a_r) \t c_i$, where $\underline{c_0} = (c_1, \ldots, c_n)$. 

A graphical representation of the isomorphism $\Psi^{-1}$ is given in Figure~\ref{figPsi}.

\vspace{0.5cm}

This morphism $\Psi$ commutes with the internal differential of $C$.

We now consider a quasi-free $\QQ$-algebra $Q_A=(\QQ(C), \partial')$ with a twisting differential $\partial'$ and explain the twisting differential $\partial''$ we obtain on $U_\QQ(Q_A) \t C$.
A graphical representation of the twisting part of the differential is given in Figure~\ref{figdiff}.

\begin{figure}
\begin{eqnarray*} 
	 \partial'' \left(\vcenter{
	\xymatrix@M=3pt@H=4pt@W=3pt@R=8pt@C=4pt{
	\diamond \ar@{-}[drr] & a_1 \ar@{-}[dr]\ar@{.}[rrr] &&& a_n\ar@{-}[dll] \\
	& & q\ar@{-}[d] & & \\
	&&&&}  }  \t c  \right) & 
	\stackrel{(def)}{=} 
	& \Psi^{-1}\left(\vcenter{
	\xymatrix@M=3pt@H=4pt@W=3pt@R=8pt@C=4pt{
	d(\partial c) \ar@{-}[drr] & a_1 \ar@{-}[dr]\ar@{.}[rrr] &&& a_n\ar@{-}[dll] \\
	& & q\ar@{-}[d] & & \\
	&&&&}  } \right)   \\
& = & 
	\Psi^{-1}\left(\vcenter{ 
 	\xymatrix@M=3pt@H=4pt@W=3pt@R=8pt@C=4pt{
 	c''_* \ar@{-}[dr]\ar@{.}[rr] & \save[]!C.[d]!C *+<36pt,10pt>\frm{}*+<6pt>\frm{(}*+<0pt>\frm{)}
\restore\ar@{}[]!L-<30pt,9pt>;[d]!L-<30pt,9pt>_(0.3){\displaystyle{\sum_{\partial'(c)} d }} & c''_*\ar@{-}[dl] \\
 	& q'\ar@{-}[drrr] &&& a_1\ar@{-}[d] \ar@{.}[rrr] &&& a_n\ar@{-}[dlll] & \\
	&&&& q\ar@{-}[d] &&&& \\
	&&&&&&&&}} 
	\right)
\end{eqnarray*}
$$ \text{where } \partial'(c)= 
\sum_{\partial(c)} \vcenter{\xymatrix@M=3pt@H=4pt@W=3pt@R=8pt@C=4pt{
c''_* \ar@{-}[dr]\ar@{.}[rr] && c''_*\ar@{-}[dl] \\
& q'\ar@{-}[d] & \\
&&}} \text{and } a_1, \ldots, a_n \in Q_A .$$

By the identity of Figure~\ref{figPsi}, the last expression can be rewritten to give:
	$$\partial'' \left(\vcenter{
	\xymatrix@M=3pt@H=4pt@W=3pt@R=8pt@C=4pt{
	\diamond \ar@{-}[drr] & a_1 \ar@{-}[dr]\ar@{.}[rrr] &&& a_n\ar@{-}[dll] \\
	& & q\ar@{-}[d] & & \\
	&&&&}  }  \t c  \right) = \sum_{\partial'(c)} \sum \vcenter{ \xymatrix@M=3pt@H=4pt@W=3pt@R=8pt@C=4pt{
	c''_*\ar@{.}[r] \ar@{-}[drr] & \diamond \ar@{-}[dr]\ar@{.}[r] & c''_*\ar@{-}[d] & a_1\ar@{.}[r]\ar@{-}[dl] & a_n\ar@{-}[dll] \\
	& & q \circ_1q'\ar@{-}[d] & & \\
		&&&&}  }   \t c''_* $$ 
\caption{\footnotesize{A graphical representation of the twisting differential in $U_\QQ(Q_A) \t C$.}}
\label{figdiff}  \end{figure}

We consider an element $\omega=q(\diamond, a_1, \ldots, a_n) \t c $ in $U_\QQ(\QQ(C)) \t C$. We compute 

$$\begin{array}{rlclc}
\partial' (\Psi(\omega)) & = & \partial' (q(dc, a_1, \ldots, a_n)) & & \\
& = &\underbrace{q (d\partial' c, a_1, \ldots, a_n)} & + & \sum_i\underbrace{ q(dc, a_1, \ldots, \partial' a_i, \ldots, a_n)}.\\
& & \partial''(\Psi(w)) & & \Psi(q(\diamond,a_1, \ldots, \partial' a_i, \ldots,a_n) \t c)\\
\end{array}$$
The second term of this sum is induced by the action of $\partial' : Q_A \fdr Q_A$. 
The image by $\Psi^{-1}$ of the first term is computed in Figure~\ref{figdiff}. We denote $\Psi^{-1}\partial''(\Psi(\omega))$ by $\partial''(\omega)$. 

To conclude, the twisting differential added to $\delta$ is the sum of $\partial' : Q_A \fdr Q_A$ and of $\partial''$ induced by the action of $\partial'$ on $C$ in $U_\QQ(Q_A) \t C$. 
There are two equivalent ways to see the module $U_\QQ(Q_A) \t C$ with its twisting differential:
$(U_\QQ(Q_A) \t C, \partial'')$ or $(U_\QQ(Q(C)) \t C, \partial'+ \partial'').$

\end{proof}
Note that we have not used the cofibrancy hypothesis on $\QQ$ in the proof of the lemma.

\vspace{1cm}

\subsection{From quasi-free $\QQ$-bimodules to resolutions of an algebra}\label{smallcomplex}

We suppose here that $\QQ$ is a $\Sigma_*$-cofibrant operad. Let $B$ be a $\QQ$-algebra.

Suppose we have a quasi-free $\QQ$-bimodule $(\QQ \circ N \circ \QQ, \partial)$ weakly equivalent to $\QQ$ as $\QQ$-bimodules, with $N$ a cofibrant $\Sigma_*$-module satisfying $N(0)=0$.

Applying Lemma \ref{RempCofib}, we get a cofibrant replacement $(\QQ \circ N \circ B, \partial')$ of $B$ in the category of $\QQ$-algebras, with $\partial'=\partial \circ_\QQ B.$

This particular cofibrant replacement allows us to compute the homology and the cohomology of $B$ as a $\QQ$-algebra using a smaller complex. We first get 
$$H_*^\QQ(B,E) = H_*(E \t_{U_{\QQ}(\QQ \circ N \circ B)} \Omega_{\QQ}(\QQ \circ N \circ B)).$$
By Lemma \ref{lemmreduc} (applied to $C= N \circ B$), this homology is identified to 
$$H_*^\QQ(B,E) = H_*(E \t N \circ B, \partial'')$$
where $\partial''$ is induced by $\partial$ in two steps explained in the proofs of Lemmas \ref{RempCofib} and \ref{lemmreduc}.

\vspace{1cm}

\subsection{An analog smaller complex for all operads}\label{analog}

Let $\PP$ be an operad and $A$ an algebra over $\PP$.

Suppose we have a quasi-free $\PP$-bimodule $(\PP \circ M \circ \PP, \partial)$ weakly equivalent to $\PP$ as a $\PP$-bimodule and such that 
\begin{enumerate}
 \item the $\Sigma_*$-module $M$ is connected and cofibrant as a $\Sigma_*$-module;
 \item the differential $\partial$ is decomposable, that is $\partial M  \subseteq \PP \circ M \circ \bar\PP + \bar\PP \circ M \circ \PP$.
\end{enumerate}

Under these hypotheses, Proposition \ref{cofbimod} implies that $(\PP \circ M \circ \PP, \partial)$ is cofibrant as a $\PP$-bimodule.

Let $Q_A = (\PP \circ M \circ A, \partial')$ be the $\PP$-algebra defined by the construction of Section \ref{smallcomplex} with the operad $\PP$ instead of the operad $\QQ$. Form the dg-module $(E \t_{U_\PP(A)} \Omega_\PP (Q_A), \partial'')$ associated to this $\PP$-algebra. We have again a map from $Q_A$ to $A$, but this map is not a weak equivalence without a cofibrancy hypothesis on $\PP$. 
Nevertheless, with the result of Lemma \ref{lemmreduc}, we can again reduce $(E \t_{U_\PP(Q_A)} \Omega_\PP (Q_A), \partial')$ to $(E \t M \circ A, \partial'')$.

Moreover, we have the following lemma of homology invariance:

\begin{sublemm}\label{HomolInvar}
 A weak equivalence of $\PP$-bimodules $(\PP \circ M_1 \circ \PP, \partial_1) \stackrel{\phi}{\fdr} (\PP \circ M_2 \circ \PP, \partial_2)$ (both satisfying the above hypotheses (1) and (2)) induces a quasi-isomorphism $(E \t M_1 \circ A, \partial_1'') \fdr (E \t M_2 \circ A, \partial_2'')$
\end{sublemm}

\begin{proof}

We consider a filtration on $(E \t M_1 \circ A, \partial_1'')$ and then use a spectral argument.

We set $F_s (E \t M_1 \circ A) = \Span_{r \leq s} \{ \xi \t m(a_1, \ldots, a_r)\}$. This complex is a subcomplex of $E \t M_1 \circ A$.

Let $\bar \phi : M_1 \fdr M_2$ denote the map $I \circ_\PP \phi \circ_\PP I$. The hypothesis (2) for $M_1$ and $M_2$ implies that $\bar \phi$ is the indecomposable part of $\phi$, and is a trivial $\Sigma_*$-cofibration. We get that $\bar\phi \circ A : M_1 \circ A \fdr M_2 \circ A$ is a trivial cofibration of dg modules.

Abusing the notation, we let $\phi$ denote also $E \t \phi \circ A : E \t M_1 \circ A \fdr E \t M_2 \circ A $.

Let us now prove that $\phi(F_s (E \t M_1 \circ A)) \subseteq F_s (E \t M_2 \circ A)$ and that $E^0\phi = E \t \bar\phi \circ A$.

\begin{displaymath}
\begin{array}{rl}
& \phi (\xi \t m(a_1, \ldots, a_r)) \\
 \stackrel{(1)}{=} & \xi \t \phi(m) (a_1, \ldots, a_r)  \\ 
 \stackrel{(2)}{=} &  \xi \t \bar\phi(m) (a_1, \ldots, a_r) + \sum \xi \t p(y_1, \ldots, y_t)(q_1,\ldots,q_s)(\underline a) \\ 
 \stackrel{(3)}{=} &  \xi \t \bar\phi(m) (a_1, \ldots, a_r) + \sum \xi \t p(y_1(\underline{q_1}(\underline{a_1}), \ldots, y_t(\underline{q_t}(\underline{a_t})))) \\
 \stackrel{(4)}{=} &  \xi \t \bar\phi(m) (a_1, \ldots, a_r) + \sum \sum_i \xi.u_i \t y_i(\underline{q_i}(\underline{a_i})).
\end{array}
\end{displaymath}

Underlined elements denote sequences of elements. Equality (2) is just using the definition of $ \bar \phi$ as the indecomposable part of $\phi$. Equality (3) comes from the composition of the subtree above each $y_i$.
In the equality (4), we use the isomorphism of Lemma \ref{lemmreduc}, and $u_i = p(y_1(\underline{q_1}(\underline{a_1})), \ldots, \diamond, \ldots,  y_t(\underline{q_t}(\underline{a_t})))$ with the hole in the $i$th position. The important thing to notice is that the arity of each $y_i$ is smaller than $r$, as the differential is decomposable.
This proves $\phi(F_s (E \t M_1 \circ A)) \subseteq F_s (E \t M_2 \circ A)$.

We now consider the associated graded complex $E^0_s (E \t M_1 \circ A) = F_s (E \t M_1 \circ A) / F_{r<s} (E \t M_1 \circ A).$
$$E^0_s (E \t M_1 \circ A) = \Span \{ \xi \t m(a_1, \ldots, a_s)\}.$$
The above calculation implies that $E^0\phi = E \t \bar\phi \circ A$.

With this equality and as $\bar\phi$ is a trivial cofibration, we get that $E^1(\phi)=H_* (E \t \bar\phi \circ A)$ is an isomorphism. 
Moreover, the spectral sequence converges, as it is a homological spectral sequence with an increasing exhaustive filtration which is bounded below.

This result implies that $H_*(\phi)$ is an isomorphism.
\end{proof}

Thus we have the following result:
\begin{sublemm}\label{HomolInvar2}
 The homology of $(E \t M \circ A, \partial'')$ does not depend on the choice of the bimodule $(\PP \circ M \circ \PP, \partial)$ weakly equivalent to $\PP$ (as a $\PP$-bimodule) such that hypotheses (1) and (2) are satisfied. 
\end{sublemm}

\begin{proof}
Suppose we have the following configuration:
\begin{equation*}
\begin{array}{c}
$\xymatrix{
(\PP \circ M_1 \circ \PP, \partial_1) \ar[dr]^{\sim} &  & (\PP \circ M_2 \circ \PP, \partial_2)\ar[dl]_{\sim} \\
 & \PP & \\
}$\end{array}
\end{equation*}

First the semi-model structure on $\PP$-bimodules gives a weak equivalence between $(\PP \circ M_1 \circ \PP, \partial_1)$ and $(\PP \circ M_2 \circ \PP, \partial_2)$ (as $(\PP \circ M_1 \circ \PP, \partial_1)$ is cofibrant). Then Lemma \ref{HomolInvar} implies that the induced arrow $(E \t M_1 \circ A, \partial_1'') \fdr (E \t M_2 \circ A, \partial_2'')$ is a quasi-isomorphism. \end{proof}

\vspace{1cm}
\subsection {Definition of $\Gamma$-homology}\label{defGamma}
Let $\PP$ be an operad, $A$ an algebra over $\PP$ and $E$ a right $U_\PP(A)$-module. 
Suppose we have a quasi-free $\PP$-bimodule $(\PP \circ M \circ \PP, \partial)$ weakly equivalent to $\PP$ as a $\PP$-bimodule, satisfying hypotheses (1) and (2) of Section \ref{analog}. 

Define the $\Gamma$-homology of the $\PP$-algebra $A$ with coefficients in $E$ to be the homology of the small complex defined in Section \ref{smallcomplex}:
$$H\Gamma^\PP_* (A,E) = H_*(E \t M \circ A, \partial'').$$

Lemma \ref{HomolInvar2} proves that the notion of $\Gamma$-homology is well defined, as it does not depend on the choice of the bimodule $(\PP \circ M \circ \PP, \partial)$.

Moreover:

\begin{theo}
Let $\QQ$ be a $\Sigma_*$-cofibrant replacement of $\PP$.\\
For $A$ a $\PP$-algebra and $E$ a right $U_\PP(A)$-module, we have $ H\Gamma^\PP_* (A,E) = H^\QQ_* (A,E).$
\end{theo}

\begin{proof} First, note that a $\PP$-algebra will also be a $\QQ$-algebra and $E$ will also be a right $U_\QQ(A)$-module.
Suppose that we are given $(\QQ \circ M \circ \QQ, \partial) \we \QQ$ a cofibrant replacement as $\QQ$-bimodules with the hypotheses above. The functor $\PP \circ_\QQ - \circ_\QQ \PP$ induces a Quillen's adjunction, and therefore we get a weak equivalence
$(\PP \circ M \circ  \PP, \partial) \we \PP$ between quasi-free $\PP$-bimodules. 
Seeing $A$ as a $\QQ$-algebra, we get $H_*^ \QQ (A,E) = H_* (E \t M \circ A, \partial'')$. 
But the right hand side is by definition $H\Gamma^\PP_* (A,E)$, as long as the differential is the same. It is the case, as both differentials are induced by the initial differential of $\QQ \circ M \circ \QQ$.\end{proof}

Thus the definition of homology by replacement of bimodules is equivalent to the natural definition by replacement of operads. Also, when the operad is $\Sigma_*$-cofibrant, we recover the usual notion of homology:

\begin{coro}\label{coro1}
Let $\QQ$ be a $\Sigma_*$-cofibrant operad, $B$ a $\QQ$-algebra and $E$ a right $U_\QQ(B)$-module.
Then $ H\Gamma^\QQ_* (B,E) = H^\QQ_* (B,E).$
\end{coro}

\vspace{1cm}

\subsection {Definition of $\Gamma$-cohomology}\label{defGammaco}
Let $\PP$ be an operad, $A$ an algebra over $\PP$ and $F$ a left $U_\PP(A)$-module. 
Suppose we have a quasi-free $\PP$-bimodule $(\PP \circ M \circ \PP, \partial)$ weakly equivalent to $\PP$ as a $\PP$-bimodule, satisfying hypotheses (1) and (2) of Section \ref{analog}. 

When $\PP$ is $\Sigma_*$-cofibrant, we can make a similar reduction of the complex $Hom_{U_\PP(Q_A)}(\Omega_{\PP}(Q_A), F)$ computing cohomology.
We take for $Q_A$ the explicit cofibrant replacement $(\PP \circ M \circ A, \partial')$ given by Lemma \ref{RempCofib}. We apply now Lemma \ref{lemmreduc} and we get $Hom_{U_\PP(Q_A)}((U_\PP(Q_A) \t M \circ A, \partial''),F)$. By adjunction, this complex is  just $(Hom_{\mathcal C} (M \circ A, F), \partial'')$. 

Following the same ideas as in Section \ref{analog}, we consider this complex even when the operad $\PP$ has no cofibrancy hypothesis.

We define the $\Gamma$-cohomology of the $\PP$-algebra $A$ with coefficients in $F$ :
$$H^*_\PP (A,F)= H^* (Hom_{\mathcal C}(M \circ A, F), \partial'').$$

A lemma similar to Lemma \ref{HomolInvar2} proves that this notion is well-defined. We recover also the usual notion of cohomology when $\PP$ is $\Sigma_*$-cofibrant.

\begin{theo}
Let $\QQ$ be a $\Sigma_*$-cofibrant replacement of $\PP$.\\
For $A$ a $\PP$-algebra and $F$ a left $U_\PP(A)$-module, we have $ H\Gamma_\PP^* (A,F) = H_\QQ^* (A,F).$
\end{theo}

\begin{coro}\label{coro2}
Let $\QQ$ be a $\Sigma_*$-cofibrant operad, $B$ a $\QQ$-algebra and $F$ a left $U_\QQ(B)$-module.
Then $ H\Gamma_\QQ^* (B,F) = H_\QQ^* (B,F).$
\end{coro}

\subsection{Remark.}
If the ground ring $\KK$ is a field of characteristic 0, then every operad $\PP$ is $\Sigma_*$-cofibrant. Hence in that case Corollary \ref{coro1} and Corollary \ref{coro2} imply that our $\Gamma$-(co)homology agrees with the standard (co)homology of $\PP$-algebras.

\vspace{2cm}

\section{Explicit complex \`a la Robinson}
From now on, we assume that $\PP$ is a connected binary (quadratic) Koszul operad.
We define an explicit $\PP$-bimodule complex, using the Koszul construction $K\PP$ and the bar construction of the symmetric group. Then we prove we can use this complex to compute $\Gamma$-homology of $\PP$-algebras. In the case $\PP=\mathsf{Com}$, we retrieve the complex introduced by Robinson.

Before defining the $\PP$-bimodules involved in the complex, we construct applications which will be needed to define the differential.

\subsection{Maps between bijections}
Let $r$ be a positive integer. Let $\underline X$ and $\underline Y$ be two ordered sets with $r$ elements.

We represent an element $w$ of $\textrm{Bij}(\underline X, \underline Y)$ by a table of values: 
\begin{equation*}
w=
\left(\begin{array}{c}
$\xymatrix@M=1pt@C=6pt@R=3pt{
x_1 & x_2 & \ldots & x_r \\
w(x_1) & w(x_2) & \ldots & w(x_r)
}$\end{array}
\right).
\end{equation*}

The ordering amounts to a fixed bijection between $\{1, \ldots, r\}$ and $\underline X$ (respectively $\underline Y)$. We can use these bijections to identify elements of $\textrm{Bij}(\underline X, \underline Y)$ with permutation of $\{1, \ldots, r \}$.

For each pair $\{i, j\} \subset \underline Y$, we form the bijection 
\begin{equation*}
c_{i,j}^e(w)=
\left(\begin{array}{c}
$\xymatrix@M=1pt@C=6pt@R=3pt{
x_1 & x_2 & \ldots &w^{-1}(i) & \ldots & \widehat{w^{-1}(j)} & \ldots & x_r \\
w(x_1) & w(x_2) & \ldots & e & \ldots & \widehat{j}  & \ldots &w(x_r)
}$\end{array}
\right)
\end{equation*}
if $w^{-1}(i)<w^{-1}(j)$ or the bijection
\begin{equation*}
c_{i,j}^e(w)=
\left(\begin{array}{c}
$\xymatrix@M=1pt@C=6pt@R=3pt{
x_1 & x_2 & \ldots &w^{-1}(j) & \ldots & \widehat{w^{-1}(i)} & \ldots & x_r \\
w(x_1) & w(x_2) & \ldots & e & \ldots & \widehat{i}  & \ldots &w(x_r)
}$\end{array}
\right)
\end{equation*}
if $w^{-1}(j)<w^{-1}(i)$.

If $w^{-1}(i)<w^{-1}(j)$, we have removed the column where $j$ is the image, and $i$ has been replaced by $e$. The application $c_{i,j}^e(w)$ is a bijection from $\underline X \smallsetminus \{w^{-1}(j)\}$ to $\underline Y \smallsetminus \{i,j\} \coprod {e}$.
In $\underline X \smallsetminus \{w^{-1}(j)\}$, we consider the restriction of the order of $\underline X$.
In $\underline Y \smallsetminus \{i,j\} \coprod {e}$, we consider the restriction of the order in $\underline Y$ with $e$ at the place of $i$.
Note that the application $c_{i,j}^e(w)$ can be identified with an element of $\Sigma_{r-1}$. 

In the case where $w^{-1}(j)<w^{-1}(i)$,  we have removed the column where $i$ is the image, and $j$ has been replaced by $e$. The application $c_{i,j}^e(w)$ is a bijection from $\underline X \smallsetminus \{w^{-1}(i)\}$ to $\underline Y \smallsetminus \{i,j\} \coprod {e}$ and can be identified with an element of $\Sigma_{r-1}$.

For each element $i$ in $\underline Y$, we form the bijection
 \begin{equation*}
c_{\emptyset,i}(w)=
\left(\begin{array}{c}
$\xymatrix@M=1pt@C=6pt@R=3pt{
x_1 & x_2 & \ldots  & \widehat{w^{-1}(i)} & \ldots & x_r \\
w(x_1) & w(x_2) & \ldots & \widehat{i}  & \ldots &w(x_r)
}$\end{array}
\right).
\end{equation*}
Here we have only removed the column where $i$ is the image.

The application $c_{\emptyset,i}(w)$ is a bijection from $\underline X \smallsetminus \{w^{-1}(i)\}$ to $\underline Y \smallsetminus \{i\}$. Again, it can be identified with an element of $\Sigma_{r-1}$ by considering the induced orders. 
These applications $c_{i,j}^e$ and $c_{\emptyset,i}$ play different roles, but note that $c_{\emptyset,i}(w)$ is just $c^y_{y,i}(w)$ for any $y$ in $\underline Y$ such that $w^{-1}(y)<w^{-1}(i)$.

\begin{sublemm}\label{cijcompatibles}
Let $\sigma$ be an element of $Bij(\underline Y) \simeq \Sigma_r$ and $w$ an element of $\textrm{Bij}(\underline X,\underline Y)$.

The applications $c_{i,j}^e(w)$ and $c_{\emptyset,i}(w)$ are compatible with the action of the symmetric group on the left, that is $\bar\sigma . c_{i,j}^e(w) = c_{\sigma(i),\sigma(j)}^e(\sigma . w) $, where $\bar\sigma$ is the bijection fixing $e$ induced by $\sigma$ on $\underline Y \smallsetminus \{i,j\} \coprod {e}$ .
\end{sublemm}

\begin{proof}
We prove the lemma in the case where $w^{-1}(i)<w^{-1}(j)$. The proof for the other case is obtained by permuting $i$ and $j$.

We already know that 
\begin{equation*}
c_{i,j}^e(w)=
\left(\begin{array}{c}
$\xymatrix@M=1pt@C=6pt@R=3pt{
x_1 &  \ldots &w^{-1}(i) & \ldots & \widehat{w^{-1}(j)} & \ldots & x_r \\
w(x_1)  & \ldots & e & \ldots & \widehat{j}  & \ldots &w(x_r)
}$\end{array}
\right).
\end{equation*}

We get 
 \begin{equation*}
\bar\sigma . c_{i,j}^e(w)=
\left(\begin{array}{c}
$\xymatrix@M=1pt@C=6pt@R=3pt{
x_1  & \ldots &w^{-1}(i) & \ldots & \widehat{w^{-1}(j)} & \ldots & x_r \\
\sigma(w(x_1))  & \ldots & e & \ldots & \widehat{j}  & \ldots & \sigma(w(x_r))
}$\end{array}
\right).
\end{equation*}

On the right hand side, we have
\begin{equation*}
\sigma . w=
\left(\begin{array}{c}
$\xymatrix@M=1pt@C=6pt@R=3pt{
x_1 &  \ldots & w^{-1}(i) & \ldots & w^{-1}(j) & \ldots & x_r \\
\sigma(w(x_1)) & \ldots & \sigma(i) & \ldots & \sigma(j)  & \ldots & \sigma(w(x_r))
}$\end{array}
\right)
\end{equation*}
and then 

 \begin{equation*}
c_{\sigma(i),\sigma(j)}^e(\sigma . w) =
\left(\begin{array}{c}
$\xymatrix@M=1pt@C=6pt@R=3pt{
x_1  & \ldots &w^{-1}(i) & \ldots & \widehat{w^{-1}(j)} & \ldots & x_r \\
\sigma(w(x_1))  & \ldots & e & \ldots & \widehat{j}  & \ldots & \sigma(w(x_r))
}$\end{array}
\right).
\end{equation*}
\end{proof}

We extend the definition of $c_{i,j}^e$ to sequences of bijections $\underline w = (w_0, \ldots w_n)$ by $ c_{i,j}^e(\underline w) = (c_{i,j}^e(w_0), \ldots, c_{i,j}^e (w_n))$.

\vspace{1cm}
\subsection{Definition of the complex}
We now define a $\Sigma_*$-module $M$ involved in our explicit complex computing $\Gamma$-homology. 
We are given $\PP$ a connected binary (quadratic) Koszul operad.

We consider $K\PP$ the Koszul construction of $\PP$, defined by $K(\PP)_{(s)}:=H_s(B_*(\PP)_{(s)})$. It is a cooperad, equipped with a differential, such that $(\PP \circ K\PP \circ \PP, \partial)$ is quasi-isomorphic to $\PP$. For more details, we refer the reader to the initial article of Ginzburg and Kapranov \cite{GK} or the article of Fresse \cite{Evanston}, of which we adopt the convention.

We also consider the chain complex $C_*(E\Sigma_\bullet)$ of the total space of the universal $\Sigma_n$-bundles in simplicial spaces, $n \in \NN$. The chain complex $C_*(E\Sigma_n)$ is the acyclic homogeneous bar construction of the symmetric group $\Sigma_n$, the module spanned in degree $t$ by the $(t+1)$-tuples of permutations $\underline w = (w_0, \ldots, w_t)$ together with the differential $\delta$ such that $\delta(\underline w) = \sum_i (-1)^i (w_0, \ldots, \widehat{w_i}, \ldots, w_t)$. We consider the left action of the symmetric group on this chain complex.

We define the $\Sigma_*$-module $M= K\PP \boxtimes C_*(E\Sigma_\bullet)$ by $M(r)= K\PP(r) \t C_*(E\Sigma_r)$. The action of the symmetric group is the diagonal action.

\vspace{1cm}
Now we construct an application $\Delta : M \fdr \PP \circ M \circ \PP$ which defines a twisting differential once extended by $\PP$-linearity on the right and as a $\PP$-derivation on the left.

Recall that the quadratic component of the cooperad product of $K\PP$ is given by the dual of the operadic composition in $\PP$ :
\begin{equation*}\begin{array}{c}
 $\xymatrix@M=3pt@H=4pt@W=3pt@R=8pt@C=4pt{
1 \ar@{-}[drr]\ar@{.}[rrrr] & \ar@{-}[dr]& \ar@{-}[d] & \ar@{-}[dl]& r \ar@{-}[dll] \\
 &  & K\PP\ar@{-}[d] &  & \\
 & &  & & 
}$\end{array}
\fdr 
 \sum \vcenter{ \xymatrix@M=4pt@H=4pt@W=4pt@R=8pt@C=4pt{
 	& j_1\ar@{-}[dr]\ar@{.}[rr] & \ar@{-}[d] & j_\ell\ar@{-}[dl] & \\
 	i_1\ar@{-}[drr] & i_2\ar@{.}[r] \ar@{-}[dr] & K\PP\ar@{-}[d]|{e}\ar@{.}[rr] & & i_k\ar@{-}[dll] \\
 	& & K\PP\ar@{-}[d] & & \\ &&&&}}
\end{equation*}
where the sum ranges over all partitions $\{i_1, \ldots, i_k \} \coprod \{j_1, \ldots, j_\ell\} = \{1, \ldots, r \}$ and $e$ is a dummy variable. 
We define two restrictions of this coproduct:
\begin{itemize}
\item $\Delta_-$ where we only keep the components of the differential where the set $\{i_1, \ldots, i_\ell \}$ is reduced to one index (when the element below in the composition is binary). 
\item $\Delta_+$ where we only keep the components of the differential where the set $\{j_1, \ldots, j_k \}$ is composed of two indices (when the element above in the composition is binary). 
\end{itemize}
Note that $\Delta_-(\gamma)= \Delta_+(\gamma)$ when $\gamma$ is an element with three inputs.

We use this coproduct $\Delta$ on $K\PP$ to define $\Delta$ on $K\PP \boxtimes C_*(E\Sigma_\bullet)$ by the following composite: 

\begin{equation*}\begin{array}{c}
 $\xymatrix@M=3pt@H=4pt@W=3pt@R=8pt@C=4pt{
1 \ar@{-}[drr]\ar@{.}[rrrr] & \ar@{-}[dr]& \ar@{-}[d] & \ar@{-}[dl]& r \ar@{-}[dll] \\
 &  & K\PP\ar@{-}[d] & \t \underline w & \\
 & &  & & 
}$\end{array}
\end{equation*}
\begin{equation*}
\fdr  \sum_i
\begin{array}{c}
 $\xymatrix@M=3pt@H=4pt@W=3pt@R=8pt@C=4pt{
1 \ar@{-}[drr]\ar@{.}[rrr] & \ar@{-}[dr] & \ar@{-}[d] & \hat{i}\ar@{-}[dl]\ar@{.}[r]& r \ar@{-}[dll] \\
i\ar@{-}[dr] &  & K\PP\ar@{-}[dl] & \t c_{\emptyset,i}(\underline w)  & \\
& K\PP\ar@{-}[d] & &  & \\
 & &  & &
}$\end{array} 
+ \sum_{\{i,j\}}
\begin{array}{c}
 $\xymatrix@M=3pt@H=4pt@W=3pt@R=8pt@C=4pt{
 & i\ar@{-}[dr] &  & j \ar@{-}[dl] & \\
1\ar@{.}[rr] \ar@{-}[drr] & \ar@{-}[dr]& K\PP\ar@{-}[d]|{e}\ar@{.}[r] & \hat{j}\ar@{-}[dl]\ar@{.}[r]& r \ar@{-}[dll] \\
 &  & K\PP\ar@{-}[d] & \t c^e_{i,j}(\underline w) & \\
 & &  & &
}$\end{array} 
\end{equation*}
\begin{equation*}
 \fdr \sum_i
\begin{array}{c}
 $\xymatrix@M=3pt@H=4pt@W=3pt@R=8pt@C=4pt{
1 \ar@{-}[drr]\ar@{.}[rrr] & \ar@{-}[dr] & \ar@{-}[d] & \hat{i}\ar@{-}[dl]\ar@{.}[r]& r \ar@{-}[dll] \\
i\ar@{-}[dr] &  & K\PP\ar@{-}[dl] & \t c_{\emptyset,i}(\underline w)  & \\
& \PP\ar@{-}[d] & &  & \\
 & &  & &
}$\end{array} 
+ \sum_{\{i,j\}} 
\begin{array}{c}
 $\xymatrix@M=3pt@H=4pt@W=3pt@R=8pt@C=4pt{
 & i\ar@{-}[dr] &  & j \ar@{-}[dl] & \\
1\ar@{.}[rr] \ar@{-}[drr] & \ar@{-}[dr]& \PP\ar@{-}[d]|{e}\ar@{.}[r] & \hat{j}\ar@{-}[dl]\ar@{.}[r]& r \ar@{-}[dll] \\
 &  & K\PP\ar@{-}[d] & \t c^e_{i,j}(\underline w) & \\
 & &  & &
}$\end{array} 
\end{equation*}

The first arrow consists in using $\Delta_-$ and $\Delta_+$ on $K\PP$ and the $c_{i,j}^e$ defined in the previous paragraph. The second arrow  comes from the twisting cochain $\kappa: K\PP \fdr \PP$ (which identifies elements of arity $2$ in $K\PP$ with elements of arity $2$ in $\PP$).

This construction defines $\Delta$ on representatives with the entries ordered from $1$ to $r$. We apply Lemma \ref{cijcompatibles} to extend this definition to $K\PP$.

\begin{sublemm}
The application $\Delta$ determines a differential of $\Sigma_*$-modules on $(\PP \circ M \circ \PP)$.
\end{sublemm}

\begin{proof}
For an element $\gamma \in K\PP(r)$ and $\underline w$ a sequence of permutations in $\Sigma_r$, we decompose $\Delta^2 (\gamma \t \underline w)$ in the sum of three terms: the part induced by $\Delta_+\Delta_+$, the part induced by $\Delta_-\Delta_-$ and the part induced by $\Delta_+\Delta_- + \Delta_-\Delta_+$.

The composite $\Delta_+\Delta_+$ yields terms of the form:
$$\begin{array}{lrl}
(I) & 
	\vcenter{
	\xymatrix@M=3pt@H=4pt@W=3pt@R=8pt@C=4pt{
	&a\ar@{-}[dr]&b\ar@{-}[d] &c\ar@{-}[dl] \\
	\ldots\ar@{-}[dr] && \kappa(\gamma'') \circ_f \kappa(\gamma''')\ar@{-}[dl]|e& \\
	& \gamma' \t c_{f,a}^e c_{b,c}^f(\underline w)\ar@{-}[d] &&& \\
	&&&&}} \\
(II) &
	\vcenter{
	\xymatrix@M=3pt@H=4pt@W=3pt@R=8pt@C=4pt{
	a\ar@{-}[dr]&&b\ar@{-}[dl]&&c\ar@{-}[dr]&&d\ar@{-}[dl] \\
	&\kappa(\gamma'')\ar@{-}[drr]|e && \ldots\ar@{-}[d] && \kappa(\gamma''')\ar@{-}[dll]|f& \\
	&&& \gamma' \t c_{a,b}^e c_{c,d}^f(\underline w)\ar@{-}[d] &&& \\
	&&&&}}&
  \end{array}$$
\begin{itemize}
 \item Let $\{i < j < k\} = \{a,b,c\}$ denote the ordered subset formed by the triple $\{a,b,c\}$ in the indexing set.
We can identify the permutation occuring in terms of the form (I):
\begin{equation*}
c_{f,a}^e c_{b,c}^f(\underline w) =
\left(\begin{array}{c}
 \xymatrix@M=1pt@C=6pt@R=3pt{
 \ldots &w^{-1}(i) & \ldots & \widehat{w^{-1}(j)} & \ldots & \widehat{w^{-1}(k)} & \ldots \\
 \ldots & e & \ldots & \widehat{j}  & \ldots &\widehat{k}  & \ldots
 }
\end{array}\right).
\end{equation*}
Thus the result of the composite $c_{a,b}^e c_{c,d}^f$ only depends on $\{i < j < k\}$.

The sum of the terms associated to a given triple $\{i < j < k\}$ is $0$ because the sum of the compositions 
	$\vcenter{
	\xymatrix@M=3pt@H=4pt@W=3pt@R=8pt@C=4pt{
	&i\ar@{-}[dr]&j\ar@{-}[d] &k\ar@{-}[dl] \\
	&& \kappa(\gamma'') \circ_f \kappa(\gamma''')\ar@{-}[d]& \\
	&&&&&}}$
cancels in $\PP$ by construction of the Koszul dual (cf. \cite[Section 5.2]{Evanston}) and the sum of terms (I) is $0$.
\item For terms (II), we have the relation $c_{a,b}^e c_{c,d}^f(\underline w)=c_{c,d}^f c_{a,b}^e (\underline w)$. By coassociativity of the coproduct in $K\PP$, the terms (II) cancel each another. Note simply that a permutation of $\kappa$ with a suspension produces a sign opposition. 
\end{itemize}
Thus the part of  $\Delta^2$ induced by $\Delta_+\Delta_+$ is $0$.

The cancellation of the part induced by $\Delta_-\Delta_-$ is similar to the proof of the cancellation of terms (I).

We now study the part induced by $\Delta_+\Delta_- + \Delta_-\Delta_+$. The composite $\Delta_-\Delta_+$ yields terms of the form:
$$\begin{array}{lrl}
(III') & 
	\vcenter{
	\xymatrix@M=3pt@H=8pt@W=3pt@R=10pt@C=4pt{
	&&b\ar@{-}[dr]&&c\ar@{-}[dl] \\
	&a\ar@{-}[dr]&&\kappa(\gamma''')\ar@{-}[dl]|f& \\
	\ldots\ar@{-}[dr] && \gamma''\t c_{\emptyset,a} c_{b,c}^f(\underline w)\ar@{-}[dl]|e& \\
	& \kappa(\gamma') \ar@{-}[d] &&& \\
	&&&&}}
  \end{array},$$
while the composite $\Delta_+\Delta_-$ yields terms of the form:
$$\begin{array}{lrl}
(III'') & 
	\vcenter{
	\xymatrix@M=3pt@H=8pt@W=3pt@R=10pt@C=4pt{
	&&b\ar@{-}[dr]&&c\ar@{-}[dl] \\
	&a\ar@{-}[dr]&&\kappa(\gamma''')\ar@{-}[dl]|f& \\
	\ldots\ar@{-}[dr] && \gamma''\t c_{b,c}^f c_{\emptyset,a} (\underline w)\ar@{-}[dl]|e& \\
	& \kappa(\gamma') \ar@{-}[d] &&& \\
	&&&&}}
  \end{array}.$$
But $c_{b,c}^f c_{\emptyset,a} (\underline w) = c_{\emptyset,a} c_{b,c}^f(\underline w)$. So by coassociativity of the coproduct in $K\PP$, we prove the cancellation of terms (III') with terms (III''). We use again that a permutation of $\kappa$ with a suspension produces a sign opposition. 

Thus the part of  $\Delta^2$ induced by $\Delta_+\Delta_- + \Delta_-\Delta_+$ is $0$.

Finally, we have proved that $\Delta^2=0$.

Moreover, the application is compatible with the symmetric action by Lemma~\ref{cijcompatibles}.
\end{proof}

\vspace{1cm}
We now consider the differential of the bar construction of the symmetric group and use it to define another differential $\delta$ on $\PP \circ M \circ \PP$.

For $\underline w = (w_0, \ldots w_n)$, recall that $\delta(\underline w) = \sum_i (-1)^i (w_0, \ldots, \widehat{w_i}, \ldots, w_n)$.

We define the application $\delta$ on $K\PP \boxtimes C_*(\Sigma_\bullet) \fdr K\PP \boxtimes C_*(\Sigma_\bullet)$ by 
$$\delta(\gamma \t \underline w) = (-1)^{|\gamma|} \gamma \t \delta(\underline w).$$

\begin{sublemm}
 The application $\delta$ induces a differential on $\PP \circ (K\PP \boxtimes C_*(E\Sigma_\bullet)) \circ \PP$ that anticommutes with $\Delta$.
\end{sublemm}
\qed

Putting all this together, we get:

\begin{theo}
 We have defined a quasi-free dg $\PP$-bimodule  $(\PP \circ (K\PP \boxtimes C_*(E\Sigma_\bullet)) \circ \PP, \Delta+\delta)$, where $\Delta$ and $\delta$ are both a differential.
\end{theo}

\vspace{1cm}
\subsection{Homology of the complex}

The goal of this paragraph is to prove that we have a quasi-isomorphism $(\PP \circ (K\PP \boxtimes C_*(E\Sigma_\bullet)) \circ \PP, \Delta+\delta) \we \PP$ of $\PP$-bimodules.

First, we consider a dg-module
morphism defined by:
\begin{equation*}
 \left\{
\begin{array}{l}
K\PP(r) \boxtimes C_0(\Sigma_r) \fdr K\PP(r)\\
K\PP(r) \boxtimes C_{\geq 1} (\Sigma_r) \fdr 0.
\end{array}
\right.
\end{equation*}
The first part of the arrow just forgets the permutation.

This morphism induces a $\PP$-bimodule morphism $(\PP \circ (K\PP \boxtimes C_*(E\Sigma_\bullet)) \circ \PP, \Delta+\delta) \we (\PP \circ K\PP \circ \PP, \partial)$, by extension by linearity on the right, and as a derivation on the left. We call $\epsilon$ this $\PP$-bimodule morphism. Note that $\Delta$ is sent to the usual differential $\partial$ of the Koszul construction with coefficients $K(\PP,\PP, \PP)$ (see \cite{Evanston} for details about that Koszul contruction), while $\delta$ is sent to 0.

We will now use a spectral argument to show that $(\PP \circ (K\PP \boxtimes C_*(E\Sigma_\bullet)) \circ \PP, \Delta+\delta)$ is quasi-isomorphic to $(\PP \circ K\PP \circ \PP, \Delta)$.

We see $(\PP \circ (K\PP \boxtimes C_*(E\Sigma_\bullet)) \circ \PP)$ as a bimodule, with differentials $\Delta$ and $\delta$. The first graduation is the bar degree $r$ in $K\PP$ and the second graduation is the number $*$ of permutations.
$$(E^0_{r,*}, d^0)= (\PP \circ (K\PP_r \boxtimes C_*(E\Sigma_\bullet)) \circ \PP, \delta)$$
We now use that $C_*(E\Sigma_\bullet)$ is acyclic, that is $H_n(C_*(E\Sigma_\bullet))= \KK$ if $n=0$ and $0$ otherwise. We also use that the functors $\PP \circ -$, $- \circ \PP$ and $K\PP_r \ \t -$ preserve quasi-isomorphims (for instance, cf.\cite[Theorem 2.1.15] {Evanston}).

Thus we get that $H_n(\PP \circ (K\PP_r \boxtimes C_*(E\Sigma_\bullet)) \circ \PP, \delta) = \PP \circ K\PP_r \circ \PP$.
$$(E^1_{r,0}, d^1)= (\PP \circ K\PP_r \circ \PP, \partial)$$
$$E^2_{r,0}=H_r(\PP \circ K\PP \circ \PP, \partial)$$
We know that the spectral sequence of a bicomplex (both graduations being bounded below) converges to the total homology of the bicomplex.

Thus $H_*(\PP \circ (K\PP \boxtimes C(E\Sigma_\bullet)) \circ \PP, \Delta+\delta)= H_*(\PP \circ K\PP \circ \PP, \partial)$.

This proves that $\epsilon$ is a quasi-isomorphism  $(\PP \circ (K\PP \boxtimes C_*(E\Sigma_\bullet)) \circ \PP, \Delta+\delta) \we  (\PP \circ K\PP \circ \PP, \partial)$. We can compose it with the quasi-isomophism between $\PP \circ K\PP \circ \PP$ and $\PP$, and finally this gives us a quasi-isomorphism $(\PP \circ (K\PP \boxtimes C_*(E\Sigma_\bullet)) \circ \PP, \Delta+\delta) \we \PP$ of $\PP$-bimodules.

\vspace{1cm}
\subsection{Back to $\Gamma$-homology}
We now prove that the $\PP$-bimodule constructed in the previous paragraphs satisfies all the required hypotheses so we can use it to compute $\Gamma$-homology.

It has the form $\PP \circ M \circ P$, with $M$ a $\Sigma_*$-module such that  $M(0)=0$. We now have to prove that $(\PP \circ (K\PP \boxtimes C_*(E\Sigma_\bullet)) \circ \PP, \Delta + \delta)$ is a cofibrant $\PP$-bimodule.

The $\PP$-bimodule $(\PP \circ (K\PP \boxtimes C_*(E\Sigma_\bullet)) \circ \PP, \Delta + \delta)$ can be seen as $(\PP \circ (K\PP \boxtimes C_*(E\Sigma_\bullet), \delta) \circ \PP, \Delta)$. The differential $\Delta$ is decomposable. We first prove that $(K\PP \boxtimes C_*(E\Sigma_\bullet), \delta)$ is a cofibrant $\Sigma_*$-module, and then Proposition \ref{cofbimod} will give us the result.

\begin{sublemm}
 The $\Sigma_*$-bimodule $(K\PP \boxtimes C_*(E\Sigma_\bullet), \delta)$ is cofibrant.
\end{sublemm}

\begin{proof}
We consider the map $f$ of $\Sigma_*$-modules $0 \fdr (K\PP \boxtimes C_*(E\Sigma_\bullet), \delta)$, which can be written as $f=(0 \t id_{\Sigma_r})_{r\in \NN}$.
According to the description of generating cofibrations in Section \ref{modcatsigma}, we have to prove that $0 \fdr (K\PP(r),0)$ is a cofibration of dg-modules. But $(K\PP(r),0)$ is assumed to be free and its differential is $0$. Hence the claim is immediate.
\end{proof}

Thus we have proved
\begin{prop}
The $\PP$-bimodule $(\PP \circ (K\PP \boxtimes C_*(E\Sigma_\bullet)) \circ \PP, \Delta + \delta)$ is cofibrant.
\end{prop}

Besides, we have seen in the previous paragraph that $(\PP \circ (K\PP \boxtimes C_*(E\Sigma_\bullet)) \circ \PP, \Delta + \delta)$ is weakly equivalent to $\PP$.

So we can use $K\PP \boxtimes C_*(E\Sigma_\bullet)$ to compute $\Gamma$-homology of algebras over $\PP$. Explicitely, we have:

\begin{theo}
 Let $\PP$ be a binary Koszul operad, $A$ an algebra over $\PP$ and $E$ a right $U_\PP(A)$-module.
$$H\Gamma^\PP_* (A,E) = H_*(E \t (K\PP \boxtimes C_*(E\Sigma_\bullet)) \circ A, \partial'')$$
where $\partial''$ is the differential induced by $\Delta + \delta$ in two steps, explained in the proofs of Lemmas \ref{RempCofib} and \ref{lemmreduc}.

Explicitly, for $x \in E$, $\gamma \in K\PP$ such that
$$\Delta_+(\gamma)= 
 \sum_{i<j} 
\begin{array}{c}
 $\xymatrix@M=3pt@H=4pt@W=3pt@R=8pt@C=4pt{
 & i\ar@{-}[dr] &  & j \ar@{-}[dl] & \\
1\ar@{.}[rr] \ar@{-}[drr] & \ar@{-}[dr]& \gamma_+''\ar@{-}[d]|{e}\ar@{.}[r] & \hat{j}\ar@{-}[dl]\ar@{.}[r]& r \ar@{-}[dll] \\
 &  & \gamma_+'\ar@{-}[d] &  & \\
 & &  & &
}$\end{array} \textrm{ and }
\Delta_-(\gamma)= \sum_{i} 
\begin{array}{c}
 $\xymatrix@M=3pt@H=4pt@W=3pt@R=8pt@C=4pt{
1 \ar@{-}[drr]\ar@{.}[rrr] & \ar@{-}[dr] & \ar@{-}[d] & \hat{i}\ar@{-}[dl]\ar@{.}[r]& r \ar@{-}[dll] \\
i\ar@{-}[dr] &  & \gamma_-''\ar@{-}[dl] &   & \\
& \gamma_-'\ar@{-}[d] & &  & \\
 & &  & &
}$\end{array}, $$
$(w_0, \ldots, w_*) \in C_*(E\Sigma_r)$ and $a_1, \ldots, a_r$ in $A$, we have:
$$\partial''(x \t \gamma \t (w_0, \ldots, w_*) \t (a_1, \ldots ,a_r)) =$$
$$\sum_{i<j} \pm x \t \gamma_+' \t (c^e_{i,j}(w_0), \ldots, c^e_{i,j}(w_*)) \t (a_1 , \ldots , \kappa(\gamma_+'')(a_i,a_j) , \ldots ,\widehat{a_j} , \ldots a_r) $$
$$+\sum_i \pm \kappa(\gamma_-')(x,a_i) \t \gamma_-'' \t (c_{\emptyset,i}(w_0), \ldots, c_{\emptyset,i}(w_*)) \t (a_1 , \ldots , \widehat{a_i} , \ldots , a_r).$$
\end{theo}

Signs are induced by the usual Koszul rule.

\subsection{Examples}
\begin{enumerate}
 \item For $\PP=\mathsf{Com}$, we have $K\PP= (\Lambda\mathsf{Lie})^{\#}$ where $\Lambda$ denotes the operadic suspension and $\#$ denotes the linear duality.
Here we retrieve easily Robinson's complex.
 \item For $\PP=\mathsf{Lie}$, we have $K\PP= (\Lambda \mathsf{Com})^\#$. We denote  $\gamma_r$ the generator of $K\PP$ in arity $r$ (it is in degree $1-r$). Let $A$ be a Lie algebra concentrated in degree $0$.
$$\partial''(x \t \gamma_r \t (w_0, \ldots, w_*) \t (a_1, \ldots ,a_r)) =$$
$$\sum_{i<j} (-1)^j x \t \gamma_{r-1} \t (c^e_{i,j}(w_0), \ldots, c^e_{i,j}(w_*)) \t (a_1 , \ldots , \kappa(\gamma_2)(a_i,a_j) , \ldots ,\widehat{a_j} , \ldots a_r) $$
$$+\sum_i (-1)^{i-1} \kappa(\gamma_2)(x,a_i) \t \gamma_{r-1} \t (c_{\emptyset,i}(w_0), \ldots, c_{\emptyset,i}(w_*)) \t (a_1 , \ldots , \widehat{a_i} , \ldots , a_r).$$
Note that we find the same signs as in the complex of Chevalley-Eilenberg.
\end{enumerate}
\vspace{1.5cm}

Similarly for the cohomology, we have the following theorem:
\begin{theo}
 Let $\PP$ be a binary Koszul operad, $A$ an algebra over $\PP$ and $F$ a left $U_\PP(A)$-module.
$$H^*_\QQ (A,F)= H^* (Hom(K\PP \boxtimes C_*(E\Sigma_\bullet) \circ A, F), \partial'').$$
where $\partial''$ is the differential induced by $\Delta + \delta$ in two steps, explained in the proofs of Lemmas \ref{RempCofib} and \ref{lemmreduc}.
\end{theo}

\vspace{1.5cm}

\section*{Acknowledgements}
I am grateful to David Chataur and Benoit Fresse for many useful discussions on the matter of this article.
I also would like to thank Joan Mill\`es for his careful reading.

\end{document}